\documentclass[reqno]{amsart}

\usepackage{amsmath,amssymb}
\usepackage{color}

\newtheorem{thm}{Theorem}[section]
\newtheorem{lemma}[thm]{Lemma}
\newtheorem{cor}[thm]{Corollary}
\newtheorem{prop}[thm]{Proposition}

\theoremstyle{definition}

\theoremstyle{remark}
\newtheorem{remark}[thm]{Remark}

\numberwithin{equation}{section}

\newenvironment{mylist}{\begin{enumerate}

}{\end{enumerate}}

\newcommand\al{\alpha}

\newcommand\ga{\gamma}
\newcommand\Ga{\Gamma}
\newcommand\de{\delta}
\newcommand\De{\Delta}
\newcommand\ep{\epsilon}
\newcommand\ka{\kappa}
\newcommand\la{\lambda}
\newcommand\La{\Lambda}

\newcommand\si{\sigma}

\renewcommand\th{\theta}

\newcommand\bal{\boldsymbol{\alpha}}

\newcommand\bfeta{\boldsymbol{\eta}}
\newcommand\calS{\mathcal S}

\newcommand\bzero{{\boldsymbol{0}}}

\newcommand\tg{\widetilde g}

\newcommand\tN{{\widetilde N}}

\newcommand\tGa{\widetilde \Gamma}
\newcommand\tla{\widetilde \lambda}
\newcommand\tpsi{\widetilde \psi}
\newcommand\tth{\widetilde \theta}

\newcommand\tw{\widetilde w}

\newcommand\R{\mathbb{R}}

\newcommand\A{{\mathcal A}}
\newcommand\C{{\mathcal C}}

\newcommand\m{{\mathfrak m}}
\newcommand\T{{\mathcal T}}

\let\le=\leqslant
\let\ge= \geqslant

\newcommand\X{\times}

\newcommand{\pa}{\partial}

\newcommand{\rmin}{r_\mathrm{min}}
\newcommand{\rmax}{r_\mathrm{max}}
\newcommand{\cmin}{c_\mathrm{min}}
\newcommand{\cmax}{c_\mathrm{max}}

\newcommand{\e}{\mathrm{e}}

\newcommand{\As}{{\mathcal A}_{a_s}}

\begin{document}

\title[Multi-point problems with variable coefficients]
{Second order, multi-point problems with variable coefficients}
\author{Fran\c cois Genoud}
\author{Bryan P. Rynne}
\address{Department of Mathematics and the Maxwell Institute
for Mathematical Sciences, Heriot-Watt University,
Edinburgh EH14 4AS, Scotland.}
\email{F.Genoud@hw.ac.uk, B.P.Rynne@hw.ac.uk}
\date{6 June 2011 \\
This work was supported by the Engineering and Physical Sciences Research Council [EP/H030514/1].}

\begin{abstract}
In this paper we consider the eigenvalue problem consisting of the
equation
\begin{equation*}
 -u'' = \la r u ,  \quad \text{on $(-1,1)$},
\end{equation*}
where $r \in C^1[-1,1], \ r>0$ and $\la \in \R$,
together with the multi-point boundary conditions
\begin{equation*}
u(\pm 1) = \sum^{m^\pm}_{i=1} \al^\pm_i u(\eta^\pm_i) ,
\end{equation*}
where
$m^\pm \ge 1$ are integers,
and, for
$i = 1,\dots,m^\pm$,
$\al_i^\pm \in \R$,
$\eta_i^\pm \in [-1,1]$,
with $\eta_i^+ \ne 1$, $\eta_i^- \ne -1$.
We show that if the coefficients
$\al_i^\pm \in \R$
are sufficiently small (depending on $r$) then the spectral properties
of this problem are similar to those of the usual separated problem,
but if the coefficients $\al_i^\pm$ are not sufficiently small then
these standard spectral properties need not hold.
The spectral properties of such multi-point problems have been obtained
before for the constant coefficient case ($r \equiv 1$), but the
variable coefficient case has not been considered previously
(apart from the existence of `principal' eigenvalues).

Some nonlinear multi-point problems are also considered.
We obtain a (partial) Rabinowitz-type result on global bifurcation from
the eigenvalues, and various nonresonance conditions for
existence of general solutions and also of nodal solutions --- these
results rely on the spectral properties of the linear problem.
\end{abstract}

\maketitle

\section{Introduction}  \label{intro.sec}

In this paper we consider the eigenvalue problem
consisting of the equation
\begin{equation} \label{eval_de.eq}
 -u'' = \la r u ,  \quad \text{on $(-1,1)$},
\end{equation}
where $r \in C^1[-1,1], \ r>0$ and $\la \in \R$,
together with the multi-point boundary conditions
\begin{equation} \label{dbc.eq}
u(\pm 1) = \sum^{m^\pm}_{i=1} \al^\pm_i u(\eta^\pm_i) ,
\end{equation}
where
$m^\pm \ge 1$ are integers,
and, for
$i = 1,\dots,m^\pm$,
$\al_i^\pm \in \R$,
$\eta_i^\pm \in [-1,1]$,
with $\eta_i^+ \ne 1$, $\eta_i^- \ne -1$.
An {\em eigenvalue} is a number $\la$ for which
\eqref{eval_de.eq}, \eqref{dbc.eq}
has a non-trivial solution $u$ (an {\em eigenfunction}).
The {\em spectrum}, $\si$,  is the set of eigenvalues.
An eigenvalue is termed {\em simple} if its algebraic multiplicity
(defined in Section~\ref{alg_mult.sec}) is equal to $1$.

For any integer $m \ge 1$ and any
$\al = (\al_1,\dots,\al_m) \in \R^m$,
the notation $\al = 0$, $\al > 0$, will mean $\al_i = 0$, $\al_i > 0$,
$i = 1,\dots,m$, respectively,
and we define the norm
\begin{align}  \label{norms.eq}
|\al|	&:=\sum_{i=1}^m |\al_i| , \quad \al \in \R^m .
\end{align}
For the coefficients in \eqref{dbc.eq} we will use the notation
$\al^\pm := (\al_1^\pm,\dots,\al_{m^\pm}^\pm) \in \R^{m^\pm}$,
$\bal := (\al^-,\al^+) \in \R^{m^-} \X \R^{m^+}$,
and similarly for $\eta^\pm$, $\bfeta$;
we also let $\bzero := (0,0) \in \R^{m^-} \X \R^{m^+}$.
For any $\ga > 0$ we define the set
$$
\A_\ga := \{ \bal : |\al^\pm| < \ga \}  .
$$
When $\bal = \bzero$ the boundary conditions \eqref{dbc.eq}
reduce to the standard, separated Dirichlet boundary
conditions  at $x=\pm 1$, so when $\bal \ne \bzero$
we call the conditions \eqref{dbc.eq}
{\em Dirichlet-type} boundary conditions.
For the separated conditions the spectral
properties of the problem are well known, and this case will play a
central role in our
analysis, with the results for  general
$\bal \in \A_\ga$ being obtained by continuation from
$\bal = \bzero$.

We will obtain various properties of the spectrum of the problem
\eqref{eval_de.eq}, \eqref{dbc.eq},
including the existence of eigenvalues, their algebraic multiplicity,
continuity properties, and the positivity of the principal
eigenfunction.
As in the classical Sturm-Liouville theory for separated boundary
conditions, the eigenfunctions will be shown to have certain
`oscillation' properties.
However, in the multi-point case these cannot
be characterised simply by counting nodal zeros of the eigenfunctions,
and instead will be described in terms of certain sets of functions
$T_k^\pm$, $k = 1,2,\dots$
(first introduced in \cite{RYN3}),
which will be defined in Section~\ref{nodal_sets.sec}.
In particular, we will prove the following theorem in
Section~\ref{evals.sec}.

\begin{thm}  \label{spec.thm}
For any $r \in C^1[-1,1], \ r>0$,
there exists $\ga =\ga(r) \in (0,1]$ such that if
$\bal\in \A_\ga$ then the spectrum $\si$ of
\eqref{eval_de.eq}, \eqref{dbc.eq}
consists of a strictly increasing sequence of simple eigenvalues
$\la_k = \la_k(r) > 0$, $k = 1,2,\dots.$
Each eigenvalue $\la_k$ has an eigenfunction $u_k \in T_k^+$.
In addition,
$\lim_{k \to \infty} \la_k = \infty.$
\end{thm}

In the constant coefficient case, with $r \equiv 1$,
it was shown in \cite[Theorem 5.1]{RYN5} that Theorem~\ref{spec.thm} is
valid with $\ga = 1$, but can fail if $\ga > 1$.
Here, in the variable coefficient case, it will be shown
that Theorem~\ref{spec.thm} is
valid for `sufficiently small' $\ga(r)$,
but may fail for values of $\ga < 1$.
The spectral properties described in Theorem~\ref{spec.thm} have not
previously been obtained for the variable coefficient problem.
Principal eigenvalues (with positive eigenfunctions) have been discussed
for the variable coefficient case in
\cite{RYN7} and \cite{WL}.

We also consider nonlinear problems of the form
\begin{equation}  \label{nonlin.eq}
 - u'' =f(\cdot,u),  \quad \text{on $(-1,1)$},
\end{equation}
together with the boundary conditions  \eqref{dbc.eq},
where $f \in C^0( [-1,1] \X \R , \R)$.
Under a suitable `nonresonance' condition on $f$ we show that
\eqref{dbc.eq}, \eqref{nonlin.eq} has a solution.
In addition, we consider the special case where equation
\eqref{nonlin.eq} has the form
\begin{equation}  \label{nodal_intro.eq}
 - u'' = g(\cdot,u)u,
\end{equation}
with $g \in C^1( [-1,1] \X \R , \R)$
(the differentiability condition on $g$ is technical and could be
removed under suitable hypotheses,
see Remark~\ref{technical.rem} below).
Since $u \equiv 0$ is a {\em trivial} solution of
\eqref{dbc.eq}, \eqref{nodal_intro.eq},
and the nonresonance result only yields existence of at least one
solution, additional arguments are required to obtain non-trivial
solutions of this problem.
In fact, we will prove the existence of {\em nodal} solutions, that is,
solutions belonging to specific sets $T_k^\pm$.
We do this by studying the bifurcation problem
\begin{equation}  \label{bif.eq}
 - u'' = \la g(\cdot,u)u.
\end{equation}
Under certain hypotheses on $g$ we prove a (partial)
Rabinowitz-type global bifurcation theorem
for the problem \eqref{dbc.eq}, \eqref{bif.eq},
showing that global (unbounded) continua of non-trivial solutions
bifurcate from the eigenvalues of the linearisation of the
bifurcation problem at $u=0$.
Nodal solutions for \eqref{dbc.eq}, \eqref{nodal_intro.eq} will then be
obtained as solutions of \eqref{dbc.eq}, \eqref{bif.eq} with $\la=1$.
This programme will be carried out using the spectral properties
of the linear problems corresponding to the asymptotes of $g$
as $u \to 0$ and $|u| \to \infty$.

In the special case where $g$ has the form
\begin{equation}  \label{product_form.eq}
g(x,u) = r(x) \tg(u),
\end{equation}
with $r \in C^1[-1,1], \ r > 0$,
and $\tg\in C^1(\R)$, we obtain nodal solutions $u \in T_k^\pm$ when
$\tg$ `crosses' an eigenvalue $\la_k=\la_k(r)$.
Such results are well known in other settings,
and similar results for autonomous multi-point problems
(with $g$ independent of $x$) were obtained in \cite{DR,RYN5}.
Nonautonomous multi-point problems similar to
\eqref{dbc.eq}, \eqref{nodal_intro.eq}
(with a separated boundary condition at one end-point)
have been discussed recently and nodal solutions were obtained in,
for example, \cite{CKK,KKW2}
(see also the references therein for other results in the same spirit).
These papers do not have the eigenvalues of the full,
variable coefficient, multi-point problem available, and they obtain
nodal solutions when $\tg$ crosses intervals between consecutive
eigenvalues of a related problem with separated boundary conditions at
both end-points.
We will describe these results further in Section~\ref{nodal.sec}
below, and compare them with our results.


Multi-point problems have received much attention recently.
For instance, in the constant coefficient case some partial results
regarding spectral properties were obtained in \cite{MOR} for a problem
with a separated boundary condition at one end-point and a multi-point
condition at the other end.
Improved results for more general problems were
then obtained in
\cite{DR,RYN3,RYN5,RYN6}.
Once the spectral properties are known they can be used to obtain
nonresonance conditions and nodal solutions in a standard manner.
The variable coefficient case has not been considered to the same
extent.
Principal eigenvalues, with positive principal eigenfunctions, have been
obtained in
\cite{RYN7,WI,WL},
and these have been used to obtain positive solutions of nonlinear
multi-point problems.
As mentioned above, the papers
\cite{CKK,KKW2}
obtain nodal solutions for multi-point problems, but they do so using
the eigenvalues of a related, separated problem, instead of those
of the multi-point problem.
For brevity, we will not discuss the background material any further
here, but simply refer the reader to the review paper \cite{GR} for
more discussion and references.


\subsection{Neumann-type boundary conditions}

If the values  of  $u$ are replaced with the values of the
derivative $u'$ in the conditions \eqref{dbc.eq} we obtain so called
{\em Neumann-type} boundary conditions.
Most of our results can be extended to deal with such conditions
(or a mixture of Dirichlet and Neumann-type conditions).
The only difficulty is that the multi-point operator introduced in
Section~\ref{func_spaces.sec} below is not invertible, since constant
functions lie in its null space.
This can be dealt with using the methods in \cite{RYN6},
which deals with such boundary conditions in the constant coefficient
case, so we will say no more about this here.
The paper \cite{KKW1} deals with a variable coefficient problem
with a separated boundary condition at one end and a Neumann-type
multi-point condition at the other end, using a similar approach  to
that in the papers \cite{CKK,KKW2}
(which deal with a Dirichlet-type condition).


\subsection{More general, nonlocal boundary conditions}

Non-local boundary conditions more general than the above multi-point
conditions have also been considered recently by several authors in
various contexts,
see for example \cite{CKK,WI} and the references therein.
For instance, the Dirichlet-type conditions \eqref{dbc.eq} can be
replaced by integral conditions of the form
\begin{equation}  \label{nonlocal}
u(\pm 1)=\int_{-1}^{1}u(y)\,d\mu_{A^{\pm}}(y),
\end{equation}
where $A^{\pm}$ are functions of bounded variations and the
corresponding measures
$\mu_{A^{\pm}}$ satisfy suitable restrictions of the form
\begin{equation}  \label{nonlocalcond}
\int_{-1}^{1}d|\mu_{A^{\pm}}|<\ga.
\end{equation}
Here, the right-hand side in \eqref{nonlocal} is a Lebesgue-Stieltjes
integral with respect to the signed measure $\mu_{A^{\pm}}$ generated by
$A^{\pm}$,
and in \eqref{nonlocalcond} the term $|\mu_{A^{\pm}}|$ denotes the total
variation of $\mu_{A^{\pm}}$
(we refer the reader to \cite[Section~19]{HS} and \cite[Section~36]{KF}
for the required measure and integration theory).

By choosing  $A^{\pm}$ to be suitable  step functions we see that
the Dirichlet-type boundary conditions
\eqref{dbc.eq} can be regarded as a special case
of the condition \eqref{nonlocal}.
Also, it is clear that \eqref{nonlocalcond} generalizes the condition
$\bal\in \A_{\ga}$
of Theorem~\ref{spec.thm}.

It is explained in \cite[Section~6]{GR}
(for equations involving the $p$-Laplacian, with constant coefficients)
how to generalize the multi-point setting to boundary conditions of the
form \eqref{nonlocal}.
The results of the present paper can be readily extended to such
boundary conditions in a similar manner.
The restriction on the coefficient $\bal$ in
Theorem~\ref{spec.thm} must be replaced by the condition
\eqref{nonlocalcond},
with a suitable $\ga=\ga(r)\in(0,1]$.
Our other results can then be obtained, with the corresponding obvious
modifications of the hypotheses.
We will not discuss this further.


\section{Preliminary results}  \label{prelim.sec}

In this section we will describe various preliminary results that will
be used in the following sections.

\subsection{Function spaces} \label{func_spaces.sec}

For any integer $n \ge 0$, let $C^n[-1,1]$ denote the usual Banach
space
of $n$-times continuously differentiable functions on $[-1,1]$,
with the usual sup-type norm, denoted by $|\cdot|_n$.
A suitable space in which to search for solutions of
\eqref{eval_de.eq},
and which incorporates the boundary conditions \eqref{dbc.eq},
is the space
\begin{align*}
X &:= \{u \in C^2[-1,1] :
\text{$u$ satisfies \eqref{dbc.eq}} \},
\\
\|u\|_X  &:= |u|_2,  \quad  u \in X.
\end{align*}
We also let
$Y:=C^0[-1,1]$, with the norm $\|\cdot \|_Y := |\cdot |_0$.

We define $\De : X \to Y$ by
\[
\De u := u'',  \quad u \in X.
\]
By the definition of the spaces $X$, $Y$, the operator
$\De$ is well-defined and continuous.
The following result is proved in  \cite[Theorem 3.1]{RYN5}.

\begin{thm}  \label{De_inverse.thm}
If $|\al^\pm| < 1$ then the operator
$\De: X \rightarrow Y$ is bijective,
and the inverse operator $\De^{-1} : Y \rightarrow X$
is continuous.
In addition,
$\De^{-1}:Y\rightarrow C^1[-1,1]$ is compact.
\end{thm}

For $h \in C^0( [-1,1] \X \R , \R)$,
we also define the Nemitskii operator $h : Y \to Y$ by
$h(u)(x) := f(x,u(x)), \ u \in Y$
(we will use the same notation for a function and its associated
Nemitskii operator --- this should cause no confusion).
The operator $h : Y \to Y$
is bounded and continuous.

\subsection{Nodal properties} \label{nodal_sets.sec}

For any $C^1$ function $u$, if $u(x_0)=0$ then $x_0$ is a
{\em simple} zero of $u$ if $u'(x_0) \ne 0$.
Now, for any integer $k \ge 1$ and any $\nu \in \{\pm\}$,
we define
$T_k^\nu \subset X$ to be the set of functions
$u \in X$ satisfying the following conditions:\\[1 ex]
(a) \  $u'(\pm 1) \ne 0$ and $\nu u'(-1) > 0$;\\
(b) \ $u'$ has only simple zeros in $(-1,1)$, and has exactly
$k$ such zeros;\\
(c) \ $u$ has a zero strictly between each consecutive zero of $u'$.
\medskip

\noindent
We also define $T_k:=T_k^+\cup T_k^-$.

\begin{remark} \label{nodal_sets.rem}
The sets $T_k^\pm$ were first introduced in \cite{RYN3} to
characterise the oscillation properties of the eigenfunctions of
multi-point Dirichlet-type problems with constant coefficients.
It was also shown in \cite{RYN3} that the usual method of characterising
the oscillations by counting the nodes of the eigenfunctions may fail
for multi-point problems.
There is a longer discussion of various methods of characterising
the oscillation properties of multi-point problems in
\cite[Section 9.4]{RYN6}.
\end{remark}

\subsection{Solution estimates} \label{soln_ests.sec}

We will now obtain various estimates on solutions of \eqref{eval_de.eq}.
Letting $r'_\pm \ge 0$ denote the positive and negative parts of $r'$,
we define
$$
r_{\min/\max}:=(\min/\max) \, r, \quad
(r'_\pm)_{\min/\max}:=(\min/\max) \, r'_\pm
$$
and
\,
$$
c_{\min} :=
  \min \Big\{ \exp\Big(-2\frac{(r'_\pm)_{\max}}{r_{\min}}\Big) \Big\},
\quad
c_{\max} := \frac{1}{c_{\min}} .
$$

For any $\la > 0$ and  $\th \in \R$ we let
$w(\la,\th) \in C^2(\R)$  denote the solution of
\eqref{eval_de.eq}
satisfying the initial conditions
\begin{equation}  \label{w_soln_form.eq}
w(\la,\th)(0) = \sin \th,
\quad
w(\la,\th)'(0) = (\la r(0))^{1/2} \cos \th.
\end{equation}
Clearly, any solution of \eqref{eval_de.eq} has the form
$u = C w(\la,\th)$, for suitable $C,\,\th \in \R$.
Also, for any solution $u$ of \eqref{eval_de.eq}
we define the Lyapunov function
\begin{equation}\label{lyap}
E(\la,u)(x):=u'(x)^2+\la r(x)u(x)^2, \quad x \in [-1,1],
\end{equation}
and when no confusion is possible we will simply write $E(x)$.
By \eqref{w_soln_form.eq},
\begin{equation} \label{E0.eq}
E(\la,w(\la,\th))(0) = \la r(0) , \quad \th \in [0,2\pi]  .
\end{equation}

\begin{lemma} \label{en_est.lem}
For $\la > 0$, if $u$ is a non-trivial solution of
\eqref{eval_de.eq} then
\begin{equation} \label{E_bnd.eq}
c_{\min} \le  \frac{E(x)}{E(0)}  \le  c_{\max} \quad \text{for all} \ x
\in [-1,1] .
\end{equation}
Hence, for  $\th \in [0,2\pi]$ and $x \in[-1,1]$,
\begin{equation}  \label{u_ud_bounds.eq}
\la r_{\min} c_{\min}  \le  w(\la,\th)'(x)^2+\la r (x)w(\la,\th)(x)^2
\le \la r_{\max} c_{\max}.
\end{equation}
\end{lemma}

\begin{proof}
From \eqref{eval_de.eq} we obtain
\begin{equation}
 -\frac{r'_-(x)}{r(x)}E(x) \le E'(x)=\la r'(x)u(x)^2 \le
\frac{r'_+(x)}{r(x)}E(x) ,
\label{en_deriv.eq}
\end{equation}
and the result follows by integration.
\end{proof}

\begin{cor}
For $\la > 0$ and $\th \in [0,2\pi]$,
\begin{equation} \label{est_u_ud.eq}
\begin{aligned}
|w(\la,\th)|_0  &\le (r_{\max} c_{\max}/r_{\min})^{1/2} ,
\\
|w(\la,\th)'|_0  &\le  (r_{\max} c_{\max})^{1/2} \la^{1/2} .
\end{aligned}
\end{equation}
\end{cor}

\begin{lemma}\label{int_est.lem}
For
$\la \ge \La_1:= (4r_{\max}c_{\max})^2/r_{\min}^3c_{\min}^2$
and $\th \in [0,2\pi]$,
\begin{equation}  \label{estuu}
\pm\int_0^{\pm1}  w(\la,\th)^2 r  \ge
c_1:=\frac14 r_{\min}c_{\min} > 0.
\end{equation}
\end{lemma}

\begin{proof}
We write $w := w(\la,\th)$ and prove the `$+$' case, the other one being
similar.
Multiplying \eqref{eval_de.eq} by $w$ and integrating by parts yields
$$
 -w(1)w'(1)+w(0)w'(0)+  \int_0^1 (w')^2 = \la  \int_0^1 w^2 r   ,
$$
and  by \eqref{u_ud_bounds.eq} we have
$$
\int_0^1 (w')^2  \ge \la r_{\min} c_{\min} - \la \int_0^1 w^2 r  .
$$
Combining these inequalities and using \eqref{est_u_ud.eq} yield
\begin{align*}
2\la  \int_0^1  w^2 r & \ge  -w(1)w'(1)+w(0)w'(0) + \la r_{\min}
c_{\min}\\
& \ge -2|w|_0|w'|_0 + \la r_{\min} c_{\min}\\
& \ge -2r_{\max}c_{\max}r_{\min}^{-1/2}\la^{1/2} + \la r_{\min} c_{\min}
,
\end{align*}
and the claim follows by a straightforward calculation.
\end{proof}

For reference, we now state a simple Lagrange identity that will be used several times below.

\begin{lemma} \label{lag.lem}
Suppose that $u,v,z \in C^2[-1,1]$, $u$ satisfies \eqref{eval_de.eq} and $z$ satisfies
\begin{equation} \label{inhom.eq}
 -z'' = \la r z + r v,  \quad \text{on $(-1,1)$}.
\end{equation}
Then
\begin{align} \label{lag.eq}
\int_0^x u v r  &= \big[u'z-uz'\big]_0^x, \quad x \in[-1,1].
\end{align}
\end{lemma}

\begin{proof}
Multiply \eqref{eval_de.eq} by $z$, \eqref{inhom.eq} by $u$, subtract
and integrate by parts.
\end{proof}

We will also need some information regarding the derivatives
$w_\la = w_\la(\la,\th)$ and
$w_\th = w_\th(\la,\th)$.
By
\eqref{eval_de.eq} and \eqref{w_soln_form.eq},
these derivatives satisfy the following initial value
problems:
\begin{alignat}{10}
\label{IVPwla}
 -w_\la'' &=\la r w_\la + r w,
\quad&
w_\la(0) &= 0, & w_\la'(0) &= \tfrac12 \la^{-1/2}
r(0)^{1/2}\cos\th,
\\
\label{IVPwth}
 -w_\th''&=\la r w_\th  ,
&
w_\th(0)&=\cos\th, \quad & w_\th'(0)&=-(\la r(0))^{1/2}\sin \th .
\end{alignat}

\begin{lemma}\label{est_v.lem}
For
$\la \ge \La_2 := \rmin^2/(4r_{\max}^3c_{\max}^2)$
and $\th \in [0,2\pi]$,
\begin{equation}  \label{est_v_vd.eq}
\begin{aligned}
\la^{1/2} |w_\la(\la,\th)|_0 &\le c_2 := 3\rmax^2(c_{\max}/\rmin)^{3/2}
\\
|w_\la(\la,\th)'|_0 &\le c_3 := 3\rmax^2c_{\max}^{3/2}/\rmin .
\end{aligned}
\end{equation}
\end{lemma}

\begin{proof}
Let $\tpsi_0 := w(\la,0)/(\la r(0))^{1/2} $,
$\psi_1 := w(\la,\pi/2)$.
Using the Lyapunov functions associated to $\tpsi_0$ and $\psi_1$
shows that
\begin{equation}  \label{estpsi.eq}
\begin{aligned}
|\tpsi_0|_0 &\le (c_{\max}/\rmin)^{1/2}\la^{-1/2}, &
  |\tpsi_0'|_0 &\le c_{\max}^{1/2}  ,
\\
|\psi_1|_0 &\le (c_{\max}\rmax/\rmin)^{1/2}, &\quad
  |\psi_1'|_0 &\le (c_{\max}\rmax)^{1/2}\la^{1/2} .
\end{aligned}
\end{equation}
The variation of constants formula for the differential
equation in \eqref{IVPwla} yields
\begin{align*}
w_\la(x) &= w_\la'(0)\tpsi_0(x) + \int_0^x
\big[\psi_1(x)\tpsi_0(y)-\psi_1(y)\tpsi_0(x)\big]r(y)w(y) \,dy,
\\
w_\la'(x) &= w_\la'(0)\tpsi_0'(x) + \int_0^x
\big[\psi_1'(x)\tpsi_0(y)-\psi_1(y)\tpsi_0'(x)\big]r(y)w(y)  \,dy,
\end{align*}
so by \eqref{IVPwla},
\begin{align*}
|w_\la|_0 &\le \frac12\rmax^{1/2} \la^{-1/2} |\tpsi_0|_0
 + 2 \rmax |w|_0 |\psi_1|_0 |\tpsi_0|_0 ,
\\
|w_\la'|_0 &\le \frac12\rmax^{1/2} \la^{-1/2} |\tpsi_0'|_0
 + \rmax |w|_0 \big(|\psi_1'|_0 |\tpsi_0|_0 + |\psi_1|_0
|\tpsi_0'|_0\big) ,
\end{align*}
and \eqref{est_v_vd.eq} now follows from
\eqref{est_u_ud.eq} and \eqref{estpsi.eq}.
\end{proof}

\begin{lemma} \label{signla.lem}
If $\la \ge \La_3 := \max \{\frac14c_1^{-2}\rmax,\La_1\}$ and
$w(\pm1)=0$
then\,$:$\\
$(a)$\ $w'(\pm1)w_\th(\pm1)>0;$\\
$(b)$\ $\pm w'(\pm1)w_\la(\pm1)>0$.
\end{lemma}

\begin{proof}
Combining  \eqref{IVPwla} and \eqref{IVPwth} with
Lemma~\ref{lag.lem} yields
\begin{equation} \label{wwth}
w'(\pm1)w_\th(\pm1)=(\la r(0))^{1/2}
\end{equation}
and
\begin{equation} \label{wwla}
w'(\pm1)w_\la(\pm1)=\int_0^{\pm1} r w^2
 - \frac12 \la^{-1/2} r(0)^{1/2}\sin \th\cos\th.
\end{equation}
Hence (a) follows immediately (for all $\la>0$),
while (b) follows from
Lemma~\ref{int_est.lem}.
\end{proof}

\subsection{The constant coefficient case}
To illustrate the above estimates, let us briefly consider the case $r
\equiv1$.
For $w=w(\la,0)$, we have
$$
w(x) = \sin(\la^{1/2}x), \quad w'(x) = \la^{1/2}\cos(\la^{1/2}x),
$$
$$
w_\la(x)=\frac12  \la^{-1/2} x \cos(\la^{1/2}x)
$$
and
$$
w_\la'(x)=\frac12\big[\la^{-1/2}\cos(\la^{1/2}x)-x
\sin(\la^{1/2}x)\big].
$$
In this case, the Lyapunov function $E \equiv 1$, and all the
above estimates hold with $c_{\min/\max}=1$.
We particularly want to emphasize here that a lower bound of the form
\eqref{estuu} cannot hold uniformly for all $\la>0$ since
$$
\pm\int_0^{\pm1} w^2  =\frac12
\left(1-\frac{\sin(2\la^{1/2})}{2\la^{1/2}}\right).
$$

\subsection{Notation}

In principle, the coefficients $\al^\pm$ are $\eta^\pm$ are fixed,
but many of the proofs will be by continuation with respect to the
coefficients $\al^\pm$ so we will often explicitly display the
dependence of various functions on these coefficients.
Of course, most of the functions we introduce also depend on the
coefficients $\eta^\pm$, but we will usually
regard these coefficients as fixed and omit them from the notation.
In particular, it will be important to observe that all estimates we
obtain will be uniform with respect to $\bal$,
as $\bal$ varies over the set
$ \{ \bal : |\al^\pm| < 1 \} . $

Also, even though the results of Theorem~\ref{spec.thm} depend on $r$,
we will generally consider $r$ to be fixed
in Sections~\ref{single_bc.sec} and \ref{evals.sec},
so, except where necessary, we will not explicitly display the
dependence on $r$.


\section{Solutions with a single boundary condition}
\label{single_bc.sec}

In this section we consider the problem
\begin{gather}
 - u'' = \la r u ,  \quad \text{on $(-1,1)$},
\label{single_bc_de.eq}
\\
u(\eta_0) =  \sum^{m}_{i=1} \al_i u(\eta_i) ,
\label{single_bc.eq}
\end{gather}
with a single, multi-point boundary condition, and fixed $\la > 0$,
$m \ge 1$, $\al \in \R^m$,
$\eta_0 \in [-1,1]$ and $\eta \in [-1,1]^m$.
We will show that the set of solutions of
\eqref{single_bc_de.eq}, \eqref{single_bc.eq},
is 1-dimensional.
This result will show that the eigenspace corresponding to an
eigenvalue is 1-dimensional.
For separated boundary value problems this 1-dimensionality is a
consequence of the uniqueness of the solutions of the initial value
problem with a single point initial condition.
Theorem~\ref{single_bc.thm} below can be regarded as an analogue
of this result for the multi-point `initial condition'
\eqref{single_bc.eq}.

For the following results, it will be convenient to regard
$\la$ as fixed, and $\al$ as variable,
so we will omit $\la$ from the notation and include $\al$.
Also, we recall that any solution $u$ of \eqref{single_bc_de.eq}
must have the form $u = C w(\th)$,
for suitable $C,\,\th \in \R$.

We first prove a preliminary lemma which will also be useful later.

\begin{lemma}  \label{single_bc_u_simple.lem}
If $|\al| < a_1:=(\rmin \cmin/\rmax \cmax)^{1/2}$
and $u$ is a non-trivial solution of
\eqref{single_bc_de.eq}, \eqref{single_bc.eq},
then
$u'(\eta_0) \ne 0. $
\end{lemma}

\begin{proof}
Suppose that $u'(\eta_0) = 0. $
Since $u = C w(\th)$, for some $C,\,\th \in \R$,
$w(\th)'(\eta_0)=0$ and
it follows by \eqref{E0.eq}-\eqref{E_bnd.eq} that
\begin{align*}
|w(\th)(\eta_0)| &=
\left( \frac{E(w(\th))(\eta_0)}{\la r (\eta_0)}\right)^{1/2}
\ge c_{\min}^{1/2} \left( \frac{r(0)}{r (\eta_0)} \right)^{1/2},
\\
|w(\th)(\eta_i)| &
\le \left(\frac{E(w(\th))(\eta_i)}{\la r(\eta_i)}\right)^{1/2}
\le c_{\max}^{1/2} \left(\frac{r(0)}{r (\eta_i)}\right)^{1/2},
\quad i=1,\dots,m.
\end{align*}
Hence, by \eqref{single_bc.eq},
$$
c_{\min}^{1/2} \left( \frac{r(0)}{r (\eta_0)} \right)^{1/2}
\le |w(\th)(\eta_0)| \le \sum^{m}_{i=1} |\al_i|  |w(\th)(\eta_i)|
\le c_{\max}^{1/2} \sum^{m}_{i=1} |\al_i|
\left(\frac{r(0)}{r(\eta_i)}\right)^{1/2},
$$
which yields
$$
\sum^{m}_{i=1} |\al_i| \left(\frac{r(\eta_0)}{r(\eta_i)}\right)^{1/2}
\ge \left(\frac{c_{\min}}{c_{\max}}\right)^{1/2}.
$$
But this contradicts the assumption that $|\al|<a_1$,
and so completes the proof.
\end{proof}

We can now prove the main result of this section.

\begin{thm}  \label{single_bc.thm}
If $|\al| < a_1$ then the set of solutions of
\eqref{single_bc_de.eq}, \eqref{single_bc.eq},
is 1-dimensional.
\end{thm}

\begin{proof}

Since any solution of \eqref{single_bc_de.eq} has
the form $u = C w(\th)$ for suitable $C,\,\th \in \R$,
we see that $u$ satisfies \eqref{single_bc.eq} if and only if
\begin{equation}  \label{single_bc_zeros.eq}
\Ga(\th,\al)  :=
w(\th)(\eta_0) - \sum^{m}_{i=1} \al_i w(\th)(\eta_i)
= 0 , \quad  \th \in \R .
\end{equation}
The function $\Ga : \R \X \R^m  \to \R$ is $C^1$,
and we will denote the partial derivative of $\Ga$ with respect to $\th$
by $\Ga_\th$.

\begin{lemma}  \label{single_bc_simple.lem}
If $|\al| < a_1$ then
$$
\Ga(\th,\al) = 0 \implies \Ga_\th(\th,\al) \ne 0 ,
\quad \th \in \R .
$$
\end{lemma}

\begin{proof}
Suppose that $\Ga(\th,\al) = \Ga_\th(\th,\al) = 0$,
for some $(\th,\al) \in \R \X \R^m$.
Defining a linear  functional $B_\al : C^0[-1,1]  \to \R$ by
$$
B_\al z := z(\eta_0) - \sum^{m}_{i=1} \al_i z(\eta_i),
\quad z \in C^0[-1,1] ,
$$
we see that
\begin{equation*}
\begin{pmatrix}
0 \\ 0
\end{pmatrix}
=
\begin{pmatrix}
B_\al(w(\th)) \\ B_\al(w_\th(\th))
\end{pmatrix}
=
\begin{pmatrix}
\cos \th & \sin \th
\\
 - \sin \th & \cos \th
\end{pmatrix}
\begin{pmatrix}
B_\al(w(0)) \\ B_\al(w(\pi/2))
\end{pmatrix} ,
\end{equation*}
so $B_\al(w(0)) = B_\al(w(\pi/2)) = 0$.
Thus, there exists $\th_0$ such that
$w(\th_0)'(\eta_0) = 0 $ and $B_\al(w(\th_0)) = 0 $.
However, this contradicts Lemma~\ref{single_bc_u_simple.lem},
which proves Lemma~\ref{single_bc_simple.lem}. \end{proof}

Now, by definition, $\Ga(\cdot,0) = w(\cdot)(\eta_0)$ has exactly 1 zero
in
$[0,\pi)$, and so it follows from Lemma~\ref{single_bc_simple.lem} and
continuity that
$\Ga(\cdot,\al)$ has exactly 1 zero in $[0,\pi)$,
for all $\al$ with $|\al| < a_1$
(of course, by periodicity and linearity,
there are other zeros outside $[0,\pi)$,
but these do not yield distinct solutions of
\eqref{single_bc_zeros.eq}).
This proves Theorem~\ref{single_bc.thm}.
\end{proof}


\section{Eigenvalues}  \label{evals.sec}

We now consider the eigenvalue problem
\eqref{eval_de.eq}, \eqref{dbc.eq},
which can be rewritten as
\begin{equation}  \label{mp_eval.eq}
 -\De(u)  = \la r u , \quad u \in X .
\end{equation}
In the following subsections we will prove various properties of the
eigenvalues and eigenfunctions, including Theorem~\ref{spec.thm}.

\subsection{Proof of Theorem~\ref{spec.thm}}  \label{proof_spec_thm.sec}

We  first establish some useful properties of solutions.
Let $(\la,u)$ be a nontrivial solution of \eqref{mp_eval.eq},
with $|\al^\pm| < 1$.

\begin{lemma}  \label{max_e_pt.lem}
 $|u(\pm 1)| \le |\al^\pm| |u|_0 <  |u|_0$.
\end{lemma}

\begin{proof}
By \eqref{dbc.eq},
$$
|u(\pm 1)| \le  \sum_{i=1} ^{m^\pm} |\al_i^\pm| |u(\eta_i^\pm)|
\le |\al^\pm|  |u|_0 .
\vspace{- 2 \baselineskip}
$$
\end{proof}
\smallskip

\begin{lemma}  \label{max_int.lem}
If $\la < 0$ then $u$ cannot have a strictly positive local max
$($respectively, a strictly negative local min$)$ in $(-1,1)$.
\end{lemma}

\begin{proof}
Suppose that $u'(x_0) = 0$, $x_0 \in (-1,1)$.
Then
$$
u'(x) =  - \la \int_{x_0}^x r u ,
\quad  x \in [-1,1] ,
$$
and the result follows by inspecting the sign of $u'$
in a neighbourhood of $x_0$.
\end{proof}

Combining Theorem~\ref{De_inverse.thm} with
Lemmas~\ref{max_e_pt.lem} and~\ref{max_int.lem}
shows that $\la > 0$
and that $|u|_0$ is attained in  $(-1,1)$,
that is, $u$ must have a global max or min in $(-1,1)$.
The following lemma shows that if $|\al^\pm|$ are sufficiently small
then $\la$ is uniformly bounded away from zero.

\begin{lemma}\label{labaway.lem}
There exists $\La_4>0$ such that
if $|\al^\pm| \le \frac12$ then $\la \ge \La_4$.
\end{lemma}

\begin{proof}
Suppose, on the contrary, that for each $n = 1,2,\dots,$ there exist
$\al_n^\pm$, with $|\al_n^\pm| \le \frac12$,
and a corresponding solution
$(\la_n,u_n)$ of
\eqref{eval_de.eq}-\eqref{dbc.eq}
with $|u_n|_0 = 1$ and $\la_n>0$,
such that  $\la_n \to 0$ as $n \to \infty$.
Then, by \eqref{eval_de.eq}, $|u_n''|_0 \to 0$,
so there exist constants $c$, $m$, and a subsequence such that
$u_n(x) \to c + m x$,
uniformly for $x \in [-1,1]$.
However, this contradicts  \eqref{dbc.eq} and the assumption that
$|\al^\pm| \le \frac12$.
\end{proof}

The next lemma shows that if $|\al^\pm|$ are sufficiently small then
any eigenfunction must lie in one of the nodal sets $T_k$.

\begin{lemma}  \label{efun_in_T.lem}
If $|\al^\pm| < a_1$ then $u'(-1)u'(1) \neq 0$
and $u \in T_k$, for some $k \ge 1$.
\end{lemma}

\begin{proof}
The first part of the lemma follows from
Lemma~\ref{single_bc_u_simple.lem}.
The second part then follows from this, together with the definition of
the sets $T_k$ and the properties of solutions of the differential
equation \eqref{eval_de.eq}.
\end{proof}

We now define $C^1$ functions
$\Ga^\pm : (0,\infty) \X \R \X \R^{m^\pm} \to \R$
by
\begin{equation}\label{defofGa}
\Ga^\pm(\la,\th,\al^\pm) :=
w(\la,\th)(\pm 1) - \sum^{m^\pm}_{i=1} \al_i^\pm w(\la,\th)(\eta_i^\pm)
\end{equation}
(where $w(\la,\th)$ is as in Section~\ref{prelim.sec}).
Substituting $w(\la,\th)$ into \eqref{dbc.eq}
shows that, for given coefficients $\al^\pm$, a number $\la$ is an
eigenvalue if and only if
\begin{equation}   \label{simul_zeros.eq}
\Ga^\pm(\la,\th,\al^\pm) = 0 ,
\end{equation}
for some $\th \in \R$.
Hence it suffices to consider the set of solutions of
\eqref{simul_zeros.eq}.

We will prove Theorem~\ref{spec.thm} by continuation with respect
to $\bal$, away from
$\bal=\bzero$, where the required information on the
solutions of \eqref{simul_zeros.eq} follows from the standard theory of
the separated Dirichlet problem,
with boundary conditions $u(\pm 1)=0$.
For reference, we state this in the following lemma.

\begin{lemma}  \label{Ga_zeros_at_zero.lem}
Suppose that $\bal=\bzero$.
For each $k = 1,2,\dots,$
there exists an eigenvalue $\la^\bzero_k > 0$ and a unique
$\th_k^\bzero \in [0,\pi)$ such that
$u^\bzero_k = w(\la^\bzero_k,\th^\bzero_k) \in T_k$
is a corresponding eigenfunction.
Also, $\la^\bzero_k < \la^\bzero_{k+1}$, and $\la^\bzero_k \to\infty$
as $k \to\infty$.
\end{lemma}

By construction, each $(\la_k^\bzero,\th_k^\bzero)$
satisfies \eqref{simul_zeros.eq}
(with $\bal = \bzero$).
Of course, by the periodicity properties of $\Ga^\pm$ with respect to
$\th$, there are other solutions of \eqref{simul_zeros.eq}
than those in Lemma~\ref{Ga_zeros_at_zero.lem},
but these do not yield distinct solutions of
the eigenvalue problem \eqref{mp_eval.eq}.
In fact, to remove these extra solutions and to reduce the domain of
$\th$ to a compact set, from now on we will regard $\th$ as lying in
the circle obtained from the interval $[0,2\pi]$ by identifying the
points $0$ and $2\pi$,
which we denote by $S^1$,
and we regard the domain of the functions $\Ga^\pm$ as
$(0,\infty) \X S^1 \X \R^{m^\pm}$.

We now wish to solve \eqref{simul_zeros.eq} for small $\bal$ by applying
the implicit function theorem at the points
$(\la^\bzero_k,\th^\bzero_k,\bzero)$, $k \ge 1$.
We would like to obtain solutions
$(\la_k(\bal),\th_k(\bal),\bal)$, for all $k \ge 1$,
for all $\bal$ lying in a fixed set $\A_\ga$
(with $\ga$ independent of $k$).
Simply applying the implicit function theorem at each of the points
$(\la^\bzero_k,\th^\bzero_k,\bzero)$
would not, of course, yield a fixed set $\A_\ga$,
so we proceed in two steps:
we first obtain a fixed set $\A_\ga$ for all
`large' eigenvalues using uniform estimates and continuation;
we then deal with the remaining (finite) set of `small' eigenvalues.

The following result will allow us to apply the implicit function
theorem at the points
$(\la^\bzero_k,\th^\bzero_k,\bzero)$,
for sufficiently large $k \ge 1$.

\begin{prop}  \label{nu_Ga_zero_zero_pos.prop}
If $\la^\bzero_k \ge \La_3$ then
\begin{equation}   \label{nu_Ga_zero_zero_pos.eq}
\pm \Ga_\la^\pm(\la^\bzero_k,\th^\bzero_k,0)
\Ga_\th^\pm(\la^\bzero_k,\th^\bzero_k,0) > 0.
\end{equation}
\end{prop}

\begin{proof}
Since $\la^\bzero_k$ is a Dirichlet eigenvalue
$w(\la^\bzero_k,\th^\bzero_k)(\pm1)=0$,
so by Lemma~\ref{signla.lem},
if $\la^\bzero_k \ge \La_3$ then
$$
\pm w_\la(\la^\bzero_k,\th^\bzero_k)(\pm 1)
w_\th(\la^\bzero_k,\th^\bzero_k)(\pm 1)
 > 0 ,
$$
which, by the definition of $\Ga^\pm$, is
\eqref{nu_Ga_zero_zero_pos.eq}.
\end{proof}

We now wish to extend the result of
Proposition~\ref{nu_Ga_zero_zero_pos.prop}
to small $\bal\neq\bzero$.
To do this we first prove a partial result.
We will suppose from now on
that $|\al^\pm| \le \min \{a_1,\frac12\}$
so that Theorem~\ref{single_bc.thm} and
Lemmas~\ref{single_bc_simple.lem},
\ref{labaway.lem} and \ref{efun_in_T.lem} hold.

\begin{lemma}  \label{Ga_simple_zeros.lem}
There exists $a_s\in \big(0,\min \{a_1,\frac12\}\big)$ and
$\La_s > 0$ such that,
if $\nu \in \{\pm\}$, $\la \ge \La_s$, $0 \le |\al^\nu| < a_s$,
and $\th \in S^1$, then
\begin{equation}   \label{Ga_ab_zero_implies.eq}
\Ga^\nu(\la,\th,\al^\nu) = 0
  \implies
 \Ga^\nu_\la(\la,\th,\al^\nu) \,
\Ga^\nu_\th(\la,\th,\al^\nu) \ne 0.
\end{equation}
\end{lemma}

\begin{proof}
We prove the result for $\Ga^+$, and we omit the superscript throughout
the proof;
the other case is similar.
We also omit the argument $(\la,\th,\al)$ throughout.

By Lemma~\ref{single_bc_simple.lem},
$\Ga=0 \Rightarrow \Ga_\th \ne  0$, so we need only prove that
$
\Ga=0 \Rightarrow \Ga_\la \ne  0.
$
Suppose, on the contrary, that there exists some $(\la,\th,\al)$ such
that
\begin{equation}  \label{Gd_zero.eq}
\Ga = w(1) - \sum^{m}_{i=1} \al_i w(\eta_i) = 0,
\quad
\Ga_\la = w_\la(1) -  \sum^{m}_{i=1} \al_i w_\la(\eta_i)  = 0 .
\end{equation}
Then
\begin{equation}  \label{estal}
|w(1)| \le |\al||w|_0 , \quad  |w_\la(1)| \le |\al||w_\la|_0.
\end{equation}
It follows from \eqref{lag.eq} with $u=v=w$ and $z=w_\la$ that
\begin{align*}
\int_0^1 w^2 r  &= [w' w_\la - w w_\la']_0^1.
\end{align*}
Inserting the estimates
\eqref{est_u_ud.eq}, \eqref{estuu}, \eqref{IVPwla}, \eqref{est_v_vd.eq}
and \eqref{estal}
into this equation shows that,
for $\la \ge \max \{\La_1,\La_2\}$,
\begin{equation}
 c_1 \le \int_0^1 w^2 r
 \le   c_4 |\al|  + \frac12 r(0)^{1/2} \la^{-1/2},
\end{equation}
where
$c_4 := (\rmax \cmax)^{1/2} ( c_2 + c_3\rmin^{-1/2} ) > 0$.
This proves that if $\la$ is sufficiently large and $|\al^\pm|$ is
sufficiently small then \eqref{Gd_zero.eq} cannot hold,
which completes the proof of Lemma~\ref{Ga_simple_zeros.lem}.
\end{proof}

We now extend the result of
Lemma~\ref{Ga_simple_zeros.lem}
to yield similar  sign conditions to those in
\eqref{nu_Ga_zero_zero_pos.eq}
in
Proposition~\ref{nu_Ga_zero_zero_pos.prop}.

\begin{prop}  \label{nu_Ga_zero_pos.prop}
If $\nu \in \{\pm\}$, $\la \ge \La_s$, $0 \le |\al^\nu| < a_s$
and $\th \in S^1$,
then
\begin{equation}   \label{nu_Ga_zero_pos.eq}
\Ga^\nu(\la,\th,\al^\nu) = 0
  \implies
\nu  \, \Ga^\nu_\la(\la,\th,\al^\nu) \,
\Ga^\nu_\th(\la,\th,\al^\nu) > 0.
\end{equation}
\end{prop}

\begin{proof}
Suppose that $\Ga^\nu(\la,\th,\al^\nu) = 0$.
We regard $(\la,\th,\al^\nu)$ as fixed,
and consider the equation
\begin{equation}  \label{Gz_eq_z.eq}
G(\tth,t) := \Ga^\nu(\la,\tth,t \al^\nu) = 0,
\quad (\tth ,\ t) \in S^1 \X [0,1].
\end{equation}
Clearly, if $t \in [0,1]$ then $|t\al^\nu| < a_s$,
so by \eqref{Ga_ab_zero_implies.eq},
\begin{equation}  \label{Gz_th.eq}
G(\th,1) = 0 \quad \text{and} \quad
G(\tth,t)  = 0 \implies G_\th(\tth,t)  \ne 0 .
\end{equation}
Hence, by \eqref{Gz_th.eq}, the implicit function theorem, and the
compactness of $S^1$,
there exists a $C^1$ solution function
$
t \to \tth(t) : [0,1] \to S^1 ,
$
for \eqref{Gz_eq_z.eq}
such that
$$
\tth(1) = \th \quad \text{and} \quad
\Ga^\nu(\la,\tth(t),t\al^\nu) = 0 , \quad t \in [0,1]
$$
(the local existence of this solution function, near $t=1$, is trivial;
standard arguments show that its domain can be extended to include the
interval $[0,1]$).

Next, by Proposition~\ref{nu_Ga_zero_zero_pos.prop}, the implication
\eqref{nu_Ga_zero_pos.eq} holds
at $(\la,\th,\al^\nu) = (\la,\tth(0),0)$
and hence, by \eqref{Ga_ab_zero_implies.eq} and continuity,
\eqref{nu_Ga_zero_pos.eq} holds at
$(\la,\tth(t),t \al^\nu)$ for all $t \in [0,1]$.
In particular, putting $t=1$ shows that \eqref{nu_Ga_zero_pos.eq} holds
at $(\la,\th,\al^\nu)$, which completes the proof
of Proposition~\ref{nu_Ga_zero_pos.prop}.
\end{proof}

We now return to the pair of equations \eqref{simul_zeros.eq}.
To solve these using the implicit function theorem we define the
Jacobian determinant
$$
J(s,\th,\bal) :=
\begin{vmatrix}
\Ga^-_\la(\la,\th,\al^-)  & \Ga^-_\th(\la,\th,\al^-)
\\[1 ex]
\Ga^+_\la(\la,\th,\al^+)  & \Ga^+_\th(\la,\th,\al^+)
\end{vmatrix},
$$
for
$(\la,\th,\bal) \in (0,\infty) \X S^1 \X \As.$
The sign properties of
$\Ga^\pm_\la,\ \Ga^\pm_\th$
proved in
Propositions~\ref{nu_Ga_zero_zero_pos.prop}
and~\ref{nu_Ga_zero_pos.prop} now yield the following results.

\begin{cor}  \label{jac.cor}
$(a)$\
If $\la^\bzero_k \ge \La_s$ then
$J(\la^\bzero_k,\th^\bzero_k,\bal) \ne 0$.
\\
$(b)$\ If $\la \ge \La_s$ and $0 \le |\al^\pm| < a_s$ then
$$
\Ga^+(\la,\th,\al^+) = \Ga^-(\la,\th,\al^-) = 0
\implies  J(\la,\th,\bal) \ne 0 .
$$
\end{cor}

The following lemma will be useful to separate `small'
and `large' eigenvalues.

\begin{lemma}  \label{s_bdd.lem}
There exists an integer $k_s$ and a constant $\ka > 0$
such that if $\la$ is an eigenvalue with eigenfunction $u \in T_k$ with
$k \ge k_s$ then $\La_s + 1 \le \la \le \ka k^2$.
\end{lemma}

\begin{proof}
This follows immediately by applying the Sturm comparison theorem
to the differential equation \eqref{eval_de.eq}.
\end{proof}

Now suppose that
$(\la,\th,\bal) \in (\La_s,\infty) \X S^1 \X \As$
is an arbitrary solution of \eqref{simul_zeros.eq},
with $w(\la,\th)\in T_{k}$ for some $k \ge k_s$.
By Corollary~\ref{jac.cor} and the implicit function theorem,
there exists a maximal open interval $\tN$ containing $1$ and a $C^1$
solution function
$$
t \to (\tla(t),\tth(t)) : \tN \to (0,\infty) \X S^1,
$$
such that
$$
(\tla(1),\tth(1)) = (\la,\th),
\quad \Ga^\pm(\tla(t),\tth(t),t\al^\pm) = 0 ,
\quad t \in \tN .
$$
Also, it follows immediately from Lemma~\ref{efun_in_T.lem}  and
continuity that
\begin{equation}  \label{kz_defn.eq}
w(\tla(t),\tth(t)) \in T_{k}  , \quad t \in \tN .
\end{equation}
Furthermore, Lemma~\ref{s_bdd.lem} and \eqref{kz_defn.eq} show that
$$
\La_s + 1 \le \tla(t) \le \ka k^2, \quad t \in \tN,
$$
and combining this estimate with Corollary~\ref{jac.cor} and
a standard continuation argument
(similar to the argument in the proof of \cite[Lemma 4.8(b)]{GR})
shows that, since $\tN$ is a maximal interval,
in fact $[0,1] \subset \tN$.

\medskip

Hence, the above arguments have shown that,
for any  $\bal \in \As$ and  $k \ge k_s$:
\begin{mylist}
\item[$C$-$(a)$]
if
$(\la,\th,\bal) \in (\La_s,\infty) \X S^1 \X  \As$
is a solution of \eqref{simul_zeros.eq}
with $w(\la,\th) \in T_k$,
then $(\la,\th,\bal)$ can be continuously connected to the solution
$(\la_k^\bzero,\th_k^\bzero,\bzero)$.
\end{mylist}
Similar arguments show that:
\begin{mylist}
\item[$C$-$(b)$]
any  solution
$(\la_k^\bzero,\th_k^\bzero,\bzero)$, $k \ge k_s$,
can be continuously connected to exactly one solution
of \eqref{simul_zeros.eq} in $(\La_s,\infty) \X S^1 \X \As$,
say
$(\la_k(\bal),\th_k(\bal),\bal) $.
\end{mylist}

We now extend these properties to the full set of solutions,
for sufficiently small $|\al^\nu|$.

\begin{prop}  \label{nu_Ga_zero_pos2.prop}
There exists $\ga\in(0,a_s]$ such that if
$\bal \in \A_{\ga}$ then the properties $C$-$(a)$ and $C$-$(b)$
hold for all $k \ge 1$.
\end{prop}

\begin{proof}

For any $\la > 0$ and  $\th \in \R$ we now let
$\tw(\la,\th) \in C^2(\R)$  denote the solution of
\eqref{eval_de.eq}
satisfying the initial conditions
\begin{equation}  \label{tw_soln_form.eq}
\tw(\la,\th)(0) = \sin \th,
\quad
\tw(\la,\th)'(0) =  \cos \th ,
\end{equation}
and define functions
$\tGa^\pm$ as in \eqref{defofGa}, but using $\tw$ instead of $w$.
Analogously to equation~\eqref{simul_zeros.eq}, eigenvalues of
\eqref{mp_eval.eq} also correspond to solutions of the pair of equations
$\tGa^\pm = 0$.
However, the fact that the initial conditions for $\tw$ do not
depend on $\la$ now allows us to establish the
analogue of Lemma~\ref{signla.lem},
and hence Proposition~\ref{nu_Ga_zero_zero_pos.prop},
for all $\la > 0$
(in this case the right hand side of \eqref{wwth}  becomes 1, while the
second term on the right hand side of \eqref{wwla}  becomes 0).
Hence, we may now follow the previous argument to show that we can apply
the implicit function theorem to the equations $\tGa^\pm = 0$
at each of the points
$(\la_k^\bzero,\th_k^\bzero,\bzero)$, $1 \le k < k_s$.
Combining this with the above results for $k \ge k_s$  then
yields a common `continuation domain' $\A_{\ga} \subset \As$
for all the eigenvalues.
\end{proof}

\begin{remark}
We used the function $w$ in the continuation argument to
obtain the `large' eigenvalues in the interval $(\La_s,\infty)$ since
this function yielded uniform estimates in Section~\ref{soln_ests.sec}
for all large $\la$,
based on the equation \eqref{E0.eq}.
Unfortunately, although the initial conditions for $\tw$
(and, more particularly, for $\tw_\la$) are simpler than those for $w$,
the analogue of \eqref{E0.eq} for $\tw$ is
$$
E(\la,\tw(\la,\th))(0) = \cos^2 \th + \la r(0) \sin^2 \th ,
\quad \th \in [0,2\pi] ,
$$
which would not yield uniform estimates for large $\la$.
\end{remark}

We conclude from the above results that, for all $k \ge 1$ and
$\bal\in \A_{\ga}$,
there is an eigenvalue $\la_k(\bal)$ with corresponding eigenfunction
$ u_k(\bal) := w(\la_k(\bal),\th_k(\bal))  \in T_k $
(or
$ u_k(\bal) := \tw(\la_k(\bal),\th_k(\bal))$
when $1 \le k < k_s$),
and there is no eigenvalue $\hat \la \ne \la_k(\bal)$
having an eigenfunction $\hat u \in T_k$.
Also, by Lemma~\ref{Ga_zeros_at_zero.lem},
$\la_k^\bzero < \la_{k+1}^\bzero$,
and by Theorem~\ref{single_bc.thm},
$\la_k(\bal) \ne \la_{k+1}(\bal)$
for any $\bal \in \A_{\ga}$,
so it follows from the continuation construction that
$\la_k(\bal) < \la_{k+1}(\bal)$
for all $\bal \in \A_{\ga}$.
The fact that $\la_k(\bal)$  has algebraic multiplicity equal to 1
(and so is simple) will be proved in Lemma~\ref{alge_mult.lem} below.

Finally, for fixed $\bal$,
the fact that $u_k(\bal) \in T_k$, for $k \ge 1$,
shows that as $k \to \infty$ the oscillation count tends to $\infty$,
so by standard properties of the differential equation
\eqref{eval_de.eq}, we must have
$\lim_{k \to\infty} \la_k = \infty$.
This concludes the proof of Theorem~\ref{spec.thm}.
\hfill $\Box$


\subsection{Continuity of the eigenvalues and eigenfunctions}
\label{continuity.sec}

The implicit function theorem construction of $\la_k$ and $u_k$ in the
proof of Theorem~\ref{spec.thm} also implies continuity properties
which will be useful below, so we state these in the following corollary
(continuity of $u_k$ will be in the space $C^0[-1,1]$, although stronger
results could easily be obtained).

\begin{cor}  \label{cts_evals.cor}
For each $k \ge 1$, $\la_k \in \R$ and $u_k \in C^0[-1,1]$ depend
continuously on
$(\bal,\bfeta) \in \A_{\ga} \X (-1,1]^{m^-} \X [-1,1)^{m^+}$.
\end{cor}


\subsection {Algebraic multiplicity  and topological degree}
\label{alg_mult.sec}

By Theorem~\ref{De_inverse.thm}, for $|\al^{\pm}|<1$
we may define a compact linear operator
$K_r : Y \to Y$ by $K_r u := -\De^{-1}(ru)$, $u \in Y$,
and then the eigenvalue problem \eqref{mp_eval.eq} is equivalent to the
problem
\begin{equation}  \label{K_la_u.eq}
u = \la K_r u,  \quad u \in Y .
\end{equation}
In particular, for $\bal\in \A_{\ga}$, the eigenvalues $\la_k(\bal)$
obtained in Theorem~\ref{spec.thm} are the characteristic values of
$K_r$,
so we may define the algebraic multiplicity of $\la_k(\bal)$ to be
the algebraic multiplicity of $\la_k(\bal)$ as a characteristic value
of $K_r$, that is
$$
\m(\la_k(\bal))
 :=\dim \bigcup_{l=1}^\infty N\big((I_Y-\la_k(\bal) K_r)^l\big),
$$
where $N$ denotes null-space and $I_Y$ is the identity on $Y$.

\begin{lemma}  \label{alge_mult.lem}
Suppose that $\bal\in \A_{\ga}$.
Then $\m(\la_k(\bal))=1$,  $k \ge 1$.
\end{lemma}

\begin{proof} The proof follows by the same continuation argument as in
\cite[Lemma~5.13]{RYN5}, starting from $\bal=\bzero$,
where the result is easily verified.
\end{proof}

Now, if $\la$ is not a characteristic value of $K_r$
then the Leray-Schauder degree $\deg(I_Y - \la K_r, B_R,0)$ is well
defined
for any $R>0$,
where $B_R$ denotes the open ball in $Y$ centred at $0$
with radius $R$,
see \cite[Ch. 13]{ZEI}.

\begin{prop}  \label{degree.prop}
Suppose that $\bal\in \A_{\ga}$.
Then, for any $R > 0$,
\[
\deg(I_Y - \la K_r, B_R,0)=
\begin{cases}
1, & \text{if $\la < \la_1,$} \\
(-1)^k, & \text{if $\la\in (\la_k, \la_{k+1}),$ $k \ge 1$.}
\end{cases}
\]
\end{prop}

\begin{proof}
The result follows from Lemma~\ref{alge_mult.lem} and the
Leray-Schauder index theorem
(see, for example, \cite[Proposition 14.5]{ZEI}).
\end{proof}


\subsection{Positivity of the principal eigenfunction}
\label{positivity.sec}

A slight extension of the above proof of Theorem \ref{spec.thm} also
proves the following result.

\begin{cor}\label{pos_efuns.cor}
$(a)$ \ If each $\al^\pm > 0$,
then $u_1 > 0$ on $[-1,1]$.
\\
$(b)$ \ If either
$\al^- < 0$ or $\al^+ < 0$,
then $u_1$ changes sign on $(-1,1)$.
\end{cor}

\begin{proof}
If $\al^\pm > 0$ then the proof of Lemma~\ref{max_e_pt.lem} can be
extended to show that $u$ cannot have a strictly positive
global max or strictly negative global min at the
end-points $x= \pm 1$.
Hence, if $\min u_1 < 0 < \max u_1$, then $u_1'$ must have at least two
zeros in $(-1,1)$, contradicting $u_1 \in T_1^+$.
Therefore,  $u_1 > 0$ on $(-1,1)$, and the boundary conditions
\eqref{dbc.eq} now show that $u(\pm1) > 0$, which proves case (a).
Case (b) follows immediately from the boundary conditions
\eqref{dbc.eq}.
\end{proof}


\subsection{Counterexamples}  \label{counterex.sec}

As mentioned in the introduction, in the constant coefficient case with
$r \equiv 1$, Theorem~\ref{spec.thm} holds for all $\bal\in\A_1$.
In this section we give two examples showing that this is not true in
the variable coefficient case.
The first example shows that, for certain  $\bal\in\A_1$ and
$r \in C^1[-1,1]$, `many' eigenvalues of
\eqref{eval_de.eq}-\eqref{dbc.eq} may be `missing'.
\medskip

\noindent{\bf Example 1.}
For $\de \in (0,1/2)$,
we consider equation \eqref{eval_de.eq} with
$$
r_\de(x) :=
\begin{cases}
1,  &  x \in [-\de,\de] ,
\\
4,  &  x \in [-1,-\de) \cup (\de,1] ,
\end{cases}
$$
together with the symmetric boundary conditions
\begin{equation}  \label{ex1_bc.eq}
u( \pm 1) = \frac{\sqrt{3}}{2} u(0).
\end{equation}
We will show that there are no eigenvalues in the interval
$I_\de := [(\tfrac{\pi}{4\de})^2 , (\tfrac{3\pi}{4\de})^2]$.
Since, for small $\de$, the interval $I_\de$ can be arbitrarily long
and an eigenvalue $\la$ corresponds to an oscillation count of `roughly'
$4\la^{1/2}/\pi$
(rigorous estimates can easily be given using the form of $r$ and
the solutions of \eqref{eval_de.eq}),
this result shows that there may be arbitrarily many oscillation counts
$k$ for which there are no corresponding eigenvalues.

We first observe that if $u$ is an eigenfunction corresponding to an
eigenvalue $\la$, then the function
$u_{\mathrm{e}}(x):=\frac12(u(x)+u(-x))$ is an even eigenfunction
corresponding to $\la$,
so it suffices to show that there are no even eigenfunctions
corresponding to eigenvalues $\la \in I_\de$.

We now search for an even solution $u$ of
\eqref{eval_de.eq}, \eqref{ex1_bc.eq},
which we suppose to be normalized to  $u(0) = 1$.
Thus, $u$ must have the form
\begin{equation}  \label{ex1_even_soln.eq}
u(x) =
\begin{cases}
\cos \la^{1/2} x,  &  x \in [0,\de] ,
\\
C_\de \cos 2\la^{1/2} (x-\ep_\de),  &  x \in (\de,1] ,
\end{cases}
\end{equation}
where $C_\de \neq 0$ and $\ep_\de$ are chosen so that
$u\in C^1[-1,1]$.
The Lyapunov function $E(\la,u)$ defined in \eqref{lyap} is piecewise
constant here, with
\begin{alignat*}{10}
u^2 + \frac{1}{\la} (u')^2 &\equiv 1 , & &\text{in $[0,\de]$,}
\\
u^2 + \frac{1}{4\la} (u')^2
&\equiv u(\de)^2 + \frac{1}{4\la} u'(\de)^2
= u(\de)^2 + \frac{1}{4} (1 - u(\de)^2)
\\&
= \frac{1}{4} + \frac{3}{4} u(\de)^2,
&\quad &\text{in $[\de,1]$}.
\end{alignat*}
Now suppose that $\la \in I_\de$.
Then $u(\de)^2\le1/2$,
and hence
$$
u(1)^2
\le \frac{1}{4} + \frac{3}{4} u(\de)^2
\le \frac{5}{8}
 < \frac{3}{4} u(0)^2,
$$
so $u$ cannot satisfy \eqref{ex1_bc.eq},
and hence there are no eigenvalues in $I_\de$.

Finally, a similar example with a $C^1$
coefficient function $r$ can be constructed by choosing a
suitable $C^1$ approximation of the  function $r_\de$.
\hfill$\square$
\medskip

Our second example shows that, for some $\bal\in\A_1$ and
$r \in C^1[-1,1]$, some eigenvalues may not be simple.

\medskip

\noindent{\bf Example 2.}
Consider equation \eqref{eval_de.eq} with
$$
r(x) := 2 - \cos(\tfrac{\pi}{2}x), \quad x\in[-1,1].
$$
Let $\mu^D>0$ be any eigenvalue of the Dirichlet problem
$$
 -u''=\mu r u  \quad \text{on $(0,1)$}, \quad u(0)=u(1)=0,
$$
with corresponding eigenfunction $u^D$.
Clearly, $u^D$ can be extended by antisymmetry to a
solution $u$ of \eqref{eval_de.eq} satisfying
\begin{equation}\label{ex_bc.eq}
u(\pm1)=\al u(0),
\end{equation}
for any $\al\in(-1,1)$.
On the other hand, let $v$ be the (even) solution of
the initial value problem
$$
 -v''=\mu^D r v \quad \text{on $(-1,1)$}, \quad v(0)=1, \ v'(0)=0.
$$
Setting $\al:=v(1)$ in \eqref{ex_bc.eq},
we see that $u$ and $v$ are two linearly independent
eigenfunctions of the problem \eqref{eval_de.eq}, \eqref{ex_bc.eq},
corresponding to the eigenvalue $\la=\mu^D$.
We only need to check that $|\al| = |v(1)| <1$.
The Lyapunov function associated with $v$ (defined in \eqref{lyap})
satisfies
$$
0 \le \frac{E'(x)}{E(x)} \le \frac{r'(x)}{r(x)} < r'(x) ,
\quad x \in (0,1],
$$
so by integration,
$$
v'(1)^2 + 2 \mu^D v(1)^2 = E(1) <  E(0) \int_0^1 r'  = \mu^D ,
$$
which yields
$|v(1)| < 2^{-1/2}$.
\hfill$\square$


\section{A non-resonance condition}
\label{nonres.sec}

In this section we consider the problem
\eqref{dbc.eq}, \eqref{nonlin.eq},
which we can rewrite as
\begin{equation}  \label{nonres.eq}
-\De u = f(u),
\quad  u \in X ,
\end{equation}
and we establish the existence of solutions under a `nonresonance'
condition.
We suppose that $f \in C^0([-1,1] \X \R,\R)$ satisfies the following
condition.
\begin{mylist}
\item[$(\mathrm{F}_\infty)$]
There exists $r_\infty \in C^1[-1,1]$ such that
\begin{equation*}  \label{f_lim_infty.eq}
\lim_{|\xi| \to \infty} \frac{f(x,\xi)}{\xi} = r_\infty(x)>0,
\quad \text{uniformly for} \ x \in[-1,1].
\end{equation*}
\end{mylist}
We also assume that
$|\al^\pm| <\ga(r_\infty)$,
and we define the eigenvalues
$\la_k^\infty := \la_k(r_\infty)$,
$k = 1,2,\dots.$

\begin{thm}\label{nonres.thm}
Suppose that $(\mathrm{F}_\infty)$ holds,
and $\la_k^\infty \ne 1$, for all $k \ge 1$.
Then \eqref{nonres.eq} has a solution.
\end{thm}

\begin{proof}
Equation \eqref{nonres.eq} is equivalent to
\begin{equation*}
u = \T(u),  \quad u \in Y,
\end{equation*}
where we define $\T:Y \to Y$ by $\T(u) := - \De^{-1}(f(u))$.
We now define a homotopy $H:[0,1]\X Y \to Y$ by
$$
H(t,u)=(1-t)\T(u) + t K_{r_\infty}(u), \quad (t,u) \in [0,1]\X Y .
$$
It follows from Theorem~\ref{De_inverse.thm} that $H$ is completely
continuous.
Furthermore, arguments similar to the proof of \cite[Theorem~5.3]{DR}
show that there exists $R>0$ such that all solutions $(t,u)\in [0,1]\X
Y$ of $u=H(t,u)$ satisfy $|u|_0<R$.
Hence, it follows from the invariance property of the degree
(see \cite[Section~13.6]{ZEI})
and Proposition~\ref{degree.prop} that
$$
\deg(I_Y-\T, B_R, 0)=\deg(I_Y-K_{r_\infty}, B_R, 0)\neq0,
$$
which proves the theorem.
\end{proof}


\section{Nodal solutions} \label{nodal.sec}

In this section we consider the following special case of
\eqref{nonres.eq},
\begin{equation}  \label{nodal.eq}
-\De u = g(u) u, \quad  u \in X ,
\end{equation}
and we establish the existence of nodal solutions of  \eqref{nodal.eq}
(that is, solutions $u$ lying in specific sets $T_k^\pm$).
We suppose throughout this section that $g\in C^1([-1,1]\X\R)$ satisfies
the
following condition.

\begin{mylist}
\item[$(\mathrm{G})$]
There exists $r_0,\ r_\infty \in C^1[-1,1]$ such that
\begin{equation*}  \label{f_lim_zero.eq}
\lim_{\xi \rightarrow 0} g(x,\xi) = r_0(x)>0 ,
\quad
\lim_{|\xi| \to \infty} g(x,\xi) = r_\infty(x)>0,
\end{equation*}
uniformly for $x \in[-1,1]$,
and
$$
 \sup |g_x| + \sup |\xi g_\xi|  <\infty.
$$
\end{mylist}
%
%
%
It follows from $(\mathrm{G})$ that the Nemitskii mapping $g:Y\to Y$
satisfies
\begin{equation}  \label{tfzero}
\lim_{|u|_0 \to 0} |g(u) - r_0 |_0 = 0 .
\end{equation}

We also assume throughout
that
$|\al^\pm| < \min \{\ga(r_0),\ga(r_\infty)\}$,
and we define the eigenvalues
$\la_k^0 := \la_k(r_0)$,
$\la_k^\infty := \la_k(r_\infty)$,
$k = 1,2,\dots.$

We will establish the existence of nodal solutions of \eqref{nodal.eq}
by studying the solution set of the bifurcation-type problem
\begin{equation}  \label{rew_bif.eq}
 - \De u = \la g(u) u, \quad  (\la,u) \in (0,\infty) \X X
\end{equation}
(solutions of \eqref{rew_bif.eq} with $\la = 1$ yield solutions
of \eqref{nodal.eq}).
Clearly, $(\la,u) = (\la,0)$ is a (trivial) solution of
\eqref{rew_bif.eq} for any $\la \in \R$,
and we let $\calS$ denote the set of non-trivial
solutions of \eqref{rew_bif.eq}.
We will use global bifurcation arguments to obtain
unbounded, closed, connected sets of solutions of \eqref{rew_bif.eq},
and then use these sets to obtain solutions of \eqref{nodal.eq}.
We first prove two preliminary lemmas.

\begin{lemma}  \label{sol_prop.lem}
For any  $\La > 0$
there exists $\ga=\ga_{g,\La}\in(0,1)$ such that,
if $\bal \in \A_\ga$ and $(\la,u) \in \calS$ with $\la \le \La$,
then $u \in T_k$ for some $k \ge 1$.
\end{lemma}

\begin{proof}

The proof is similar to that of Lemma~\ref{efun_in_T.lem}, and we need
only
show that $u'(-1)u'(1)\neq0$.
For $(\la,u) \in \calS$, we define the Lyapunov function
\begin{equation*}
E(x):=u'(x)^2 + \la g(x,u(x))u(x)^2, \quad x \in[-1,1].
\end{equation*}
Then, by  \eqref{rew_bif.eq},
$$
E'(x)=\la [ g_x(x,u(x)) + g_\xi(x,u(x))u'(x) ] u(x)^2$$
and so, by $(\mathrm{G})$, there exist $C_1,C_2>0$ such that
\begin{align*}
|E'| 	& \le C_1\la u^2 + C_2 \la^{1/2} |u'| \la^{1/2} |u|
  \le C_1\la u^2 + C_2 \la^{1/2} [(u')^2 + \la u^2] \\
  & \le (C_1 + C_2 \la^{1/2})[(u')^2 + \la u^2]
  = (C_1 + C_2 \la^{1/2})[(u')^2 + g_{\min}^{-1} \la g_{\min} u^2] \\
  & \le \max\{1,g_{\min}^{-1}\}
  (C_1 + C_2 \la^{1/2})[(u')^2 + \la g_{\min} u^2] \\
  & \le \max\{1,g_{\min}^{-1}\} (C_1 + C_2 \la^{1/2}) E =: R(\la) E.
\end{align*}
It follows by integration that
$$
E(0)\e^{-R(\la)}  \le E(x) \le  E(0)\e^{R(\la)} .
$$
A similar argument to the proof of
Lemma~\ref{single_bc_u_simple.lem} now shows that if
$$
|\al^\pm|<
\ga_{g,\La}:=\left(\frac{g_{\min}}{g_{\max}}\right)^{1/2} e^{-R(\La)}
$$
then $u'(-1)u'(1)\neq0$.
\end{proof}

\begin{lemma}  \label{sturm_comp.lem}
If  $(\la,u) \in \calS$ with $u \in T_k$, for some $k \ge 1$,
then  $\la < \La(k) := g_{\min}^{-1}((k+2)\pi/2)^2$.
\end{lemma}

\begin{proof}
The function $\sin \big( \frac{(k+2)\pi}{2} (x+1)\big)$ has
$k+3$ zeros in $[-1,1]$
so, by the Sturm comparison theorem,
if $\la \ge \La(k)$ then $u$ must have at least
$k+2$ zeros in $[-1,1]$,
and hence  $u'$ must have at least
$k+1$ zeros in $[-1,1]$,
which contradicts $u\in T_k$.
\end{proof}

We can now prove the following Rabinowitz-type
global bifurcation theorem, for individual nodal counts $k \ge 1$.

\begin{thm}  \label{rab_glob.thm}
For a given $k \ge 1$, let $\ga_k := \ga_{g,\La(k)}$
$($recall Lemmas~\ref{sol_prop.lem} and~\ref{sturm_comp.lem}$)$,
and suppose that  $\bal\in \A_{\ga_k}$.
Then there exists closed connected sets
$\C_k^\pm \subset \overline \calS$
with the properties\,$:$
\begin{mylist}
\item
$(\la_k^0,0)\in {\mathcal C_k^\pm};$
\item
${\mathcal C_k^\pm}\backslash \{(\la_k^0,0)\} \subset
(0,\La(k)] \times T_k^\pm;$
\item
${\mathcal C_k^\pm}$ is unbounded in $(0,\La(k)] \X Y$.
\end{mylist}
\end{thm}

\begin{proof}
By Theorem~\ref{De_inverse.thm}, equation \eqref{rew_bif.eq}
is equivalent to the equation
\begin{equation}  \label{invertbif.eq}
 u = G(\la,u),  \quad (\la,u) \in (0,\infty) \X Y ,
\end{equation}
where
$$
G(\la,u) := -\la \De^{-1}(g(u) u) ,  \quad (\la,u) \in (0,\infty) \X Y ,
$$
and by  \eqref{tfzero},
$$
|G(\la,u) - \la K_{r_0} u|_0 = {\rm o}(|u|_0) ,
\quad \text{as $|u|_0 \to 0$,}
$$
uniformly on bounded $\la$ intervals.
This shows that equation \eqref{invertbif.eq} has the same
form as the bifurcation problems considered in \cite{RAB}.
Hence, by the methods in \cite{RAB}
(in particular, the proof of
\cite[Theorem 2.3]{RAB}, which deals with separated Sturm-Liouville
problems),
we can show that:
\begin{itemize}
\item
there exist closed, connected sets $\C_k^\pm \subset \overline\calS$
satisfying  (a) and (b)
in a neighbourhood of $(\la_k^0,0)$,
and the following alternatives for each $\nu\in \{\pm\}$:
(i)~$(\la_{k'}^0,0) \in \C_k^\nu$, with $k' \ne k$,
or
(ii)~$\C_k^\nu$ is unbounded in $(0,\infty) \X Y$;
\item
by Lemmas~\ref{sol_prop.lem} and~\ref{sturm_comp.lem},
the sets $\C_k^\pm$ cannot intersect either
$(0,\La(k)] \times \pa T_k^\pm$ or $\{\La(k)\} \times T_k^\pm$,
so (b) must hold globally;
\item
since (b) holds, it now follows,
as in the proof of \cite[Theorem 2.3]{RAB},
that alternative (i) above cannot hold,
which proves (c).
\end{itemize}
The proof is complete.
\end{proof}

\begin{remark}\label{technical.rem}
For fixed $\bal\neq\bzero$,
Theorem~\ref{rab_glob.thm} does not yield unbounded
continua $\C_k^\pm$ for arbitrarily large $k$,
since $\lim_{k \to\infty} \ga_k = 0$.
This is due to the dependence of $\ga_{g,\La}$ on $\La$ in
Lemma~\ref{sol_prop.lem}.
Also, the proof of this lemma is where we used the assumptions in
condition $(\mathrm{G})$ on the derivatives of $g$.
By a more detailed argument we could remove these assumptions and
obtain Lemma~\ref{sol_prop.lem} for (suitably bounded) $C^0$ functions
$g$.
However, it is not clear that the dependence on $\La$ can be removed,
or what would be the most general result, so for simplicity
we have retained condition $(\mathrm{G})$ and the above method of proof
of Lemma~\ref{sol_prop.lem}, which fits in with our previous arguments
using the Lyapunov functions.
\end{remark}

We now prove the existence of nodal solutions for \eqref{nodal.eq}.

\begin{thm} \label{nodal.thm}
Suppose that
$(\la_k^0 - 1) (\la_k^\infty - 1) < 0$
and $\bal \in \A_{\ga_k}$, for some $k \ge1$.
Then \eqref{nodal.eq} has solutions $u_k^\pm \in T_k^\pm$.
\end{thm}

\begin{proof}
By Theorem~\ref{rab_glob.thm}~(c) and standard arguments
(see, for example, the proof of \cite[Theorem~5.3]{DR}),
we may choose sequences
$\{(\mu_n^\pm,u_n^\pm)\} \subset \C_k^\pm$
such that
$\mu_n^\pm \to \la_k^\infty$ and $|u_n^\pm|_0\to\infty$, as
$n \to\infty$.
Hence, by connectedness, the sets $\C_k^\pm$ intersect the hyperplane
$\{1\} \X Y$,
which proves Theorem~\ref{nodal.thm}.
\end{proof}

Finally, we briefly consider the special case where $g$ has the form
\eqref{product_form.eq}
with $r \in C^1[-1,1], \ r > 0$,
and $\tg\in C^1(\R)$.
The hypothesis $(\mathrm{G})$ is satisfied if
$|\tg'(\xi)\xi|$ is bounded and the limits
\begin{equation*}  \label{g_lim.eq}
\tg_0:=\lim_{\xi \rightarrow 0} \tg(\xi) > 0,
\quad
\tg_\infty:=\lim_{|\xi| \rightarrow \infty} \tg(\xi) > 0
\end{equation*}
exist.
For $k \ge 1$, if we write $\la_k:=\la_k(r)$ then
$\la_k^0 = \la_k/\tg_0$,
$\la_k^\infty = \la_k/\tg_\infty$,
and Theorem~\ref{nodal.thm} now takes the following form.

\begin{thm}\label{nodal2.thm}
Suppose that
$(\la_k-\tg_0)(\la_k-\tg_\infty)<0$
and $\bal \in \A_{\ga_k}$, for some $k \ge1$.
Then \eqref{nodal.eq} has  solutions
$u_k^\pm \in T_k^\pm$.
\end{thm}

The condition
$(\la_k-\tg_0)(\la_k-\tg_\infty)<0$
in Theorem~\ref{nodal.thm} says that the range of $\tg$ `crosses' the
eigenvalue $\la_k$.
The use of such crossing conditions to obtain solutions is well
known in many settings.

Similar results to Theorem~\ref{nodal2.thm} have been obtained recently
in, for example, \cite{CKK,KKW2} and references cited therein.
Roughly speaking, these papers deal with problems of the form
\eqref{nodal.eq} with boundary conditions of the form
\begin{equation}  \label{half_sep_bc.eq}
c_0 u(-1) + c_1 u'(-1) = 0,  \quad
u(1) = \sum^{m^+}_{i=1} \al^+_i u(\eta^+_i) ,
\end{equation}
with constants $c_0,\,c_1$ satisfying $|c_0| + |c_1| > 0$,
that is,  a single, separated boundary condition at one end-point, and a
multi-point condition at the other end-point.
They also assume that $g$ has the form \eqref{product_form.eq}
(or a sum of such terms --- to simplify the discussion we merely
consider the form \eqref{product_form.eq} here).
Although our results, as stated, do not cover the  multi-point boundary
conditions \eqref{half_sep_bc.eq}, it is trivial to extend them to
do so.
In particular, the analogue of Theorem~\ref{spec.thm} holds here.
The proof is by a similar (but simpler) continuation argument in which
the function $w$ can now be specified to satisfy
the boundary condition at $x = -1$ for all $\la$ and $\bal$
(so there is now no analogue of the variable $\th$),
and then the boundary condition at $x = 1$ leads to a single equation to
solve for $\la(\bal)$
(instead of a pair of equations).
This considerably simplifies the construction, and yields a set of
eigenvalues $\la_k(r)$, $k \ge 1$, of
\eqref{eval_de.eq}, \eqref{half_sep_bc.eq},
for all $\bal$ in a suitable set $\A_{\ga(r)}$.
Once the eigenvalues have been obtained the discussion of
the nonlinear problem
\eqref{nodal.eq}, \eqref{half_sep_bc.eq}
proceeds as above and, in particular, yields the
analogue of Theorem~\ref{nodal2.thm} for this problem.

The main results and hypotheses in \cite{CKK,KKW2} for the problem
\eqref{nodal.eq}, \eqref{half_sep_bc.eq}
can be summarised as follows
(for brevity we omit various detailed conditions;
in particular, it is hard to describe the conditions on the size of
$|\bal|$ in \cite{CKK,KKW2}, or to compare them with our conditions).
Letting $\mu_k$, $k \ge 1$, denote the eigenvalues of \eqref{eval_de.eq}
together with the separated boundary conditions
\begin{equation}  \label{full_sep_bc.eq}
c_0 u(-1) + c_1 u'(-1) = 0,  \quad u'(1) = 0 ,
\end{equation}
they obtain nodal solutions $u_k^\pm\in T_k^\pm$
of
\eqref{nodal.eq}, \eqref{half_sep_bc.eq}
if
$\tg_0 < \mu_k < \mu_{k+1} < \tg_\infty$
or
$\tg_\infty < \mu_k < \mu_{k+1} < \tg_0$,
that is,
if $\tg$ `crosses' the interval
$[\mu_k,\mu_{k+1}]$.
In other words, the range of $\tg$ is compared
with the eigenvalues of the linear problem obtained by replacing the
original multi-point boundary conditions \eqref{half_sep_bc.eq} with the
separated boundary conditions \eqref{full_sep_bc.eq}.
Heuristically, one could say that using the eigenvalues $\mu_k$
corresponding to the changed boundary conditions leads to the necessity
for the function $\tg$ to cross the interval $[\mu_k,\mu_{k+1}]$, rather
than crossing the single eigenvalue $\la_k$ obtained from the original
boundary conditions.
The following lemma shows that
if $\tg$ crosses $[\mu_k,\mu_{k+1}]$ then
it crosses $\la_k$, but obviously the converse need not hold.

\begin{lemma}
If $\bal \in \A_{\ga(r)}$ then, for each $k \ge 1$,
$\mu_k < \la_k(\bal) < \mu_{k+1}$.
\end{lemma}

\begin{proof}
We denote by $u_k(\bal)$ the eigenfunctions of problem
\eqref{eval_de.eq}, \eqref{half_sep_bc.eq} obtained by adapting
our arguments in Section~\ref{evals.sec}, as outlined above.
First, at $\bal=\bzero$, we apply Sturm's comparison theorem
to the problems \eqref{eval_de.eq}, \eqref{half_sep_bc.eq}
and
\eqref{eval_de.eq}, \eqref{full_sep_bc.eq}. It follows that
$\mu_k < \la_k(\bzero) < \mu_{k+1}$, for all $k\ge1$.
Now suppose that $\la_k(\bal)=\mu_k$, for some $k \ge 1$ and
$\bal \in \A_{\ga(r)}$,
and let $v_k$ be an eigenfunction of
\eqref{eval_de.eq}, \eqref{full_sep_bc.eq},
corresponding to the eigenvalue $\mu_k$.
Since $u_k(\bal)$ and $v_k$ each satisfy the
boundary condition in \eqref{full_sep_bc.eq} at $x=-1$,
by a rescaling we may suppose that $u_k(\bal) = v_k$.
But the analogue of Lemma~\ref{efun_in_T.lem} for this setting
shows that $v_k'(1) \ne 0$,
which contradicts $v_k$ being an eigenfunction of
\eqref{eval_de.eq}, \eqref{full_sep_bc.eq}.
Therefore, for all $\bal \in \A_{\ga(r)}$ and $k\ge1$,
$\la_k(\bal)\neq\mu_k$
and a similar argument shows that $\la_k(\bal)\neq\mu_{k+1}$.
Hence, by continuation,  $\mu_k<\la_k(\bal)<\mu_{k+1}$
for all $\bal \in \A_{\ga(r)}$, $k\ge1$.
\end{proof}

We conclude that the crossing condition in \cite{CKK,KKW2} is more
restrictive than that in  Theorem~\ref{nodal2.thm} above.
In addition, we have obtained results for the boundary conditions
\eqref{dbc.eq}, having multi-point conditions at both end-points.
It is much easier to deal with the conditions \eqref{half_sep_bc.eq}
than with \eqref{dbc.eq}
since shooting methods can be used, as in \cite{CKK,KKW2}
(shooting from $x=-1$, using the separated boundary condition to
provide an `initial condition';
this is not possible with multi-point conditions at both ends).
Finally, as explained in the introduction, it is straightforward to
extend our methods to the case of integral boundary conditions at both
end-points.

\end{document}